\documentclass{article}
\usepackage{pspicture} 

\usepackage{amsfonts}
\usepackage{amsmath}
\usepackage{amscd}

\usepackage{pst-node}

\usepackage{amssymb}

\newtheorem{lem}{Lemma}[section]
\newtheorem{theo}[lem]{Theorem}
\newtheorem{coro}[lem]{Corollary}
\newtheorem{propo}[lem]{Proposition}
\newtheorem{rema}[lem]{Remark}
\newtheorem{defi}[lem]{Definition}
\newtheorem{ex}[lem]{Example}
\newtheorem{Question}[lem]{Question}

\newenvironment{proof}{\paragraph*{Proof}}
{\par}

\newcommand\GL{{\mathrm{GL}}}

\newcommand\PSL{{\mathrm{PSL}}}
\newcommand\PGL{{\mathrm{PGL}}}
\newcommand\GSP{{\mathrm{GSp}}}

\newcommand\End{{\mathrm{End}}}
\newcommand\Gal{{\mathrm{Gal}}}
\newcommand\Spec{{\mathrm{Spec}}}

\newcommand\eps\varepsilon
\newcommand\ph\varphi

\newcommand\N{{\mathbb N}}
\newcommand\T{{\mathbb T}}

\newcommand\F{{\mathbb F}}
\newcommand\Q{{\mathbb Q}}
\newcommand\PPP{{\mathbb P}}
\newcommand\Z{{\mathbb Z}}

\newcommand\gerP{{\mathfrak P}} 
\newcommand\gerA{{\mathfrak A}} 
\newcommand\gerS{{\mathfrak S}}

\title{On uniform large Galois images for modular abelian varieties}

\author{E. Ghate, P. Parent}

1

\begin{document}

\maketitle 
\begin{abstract}
We formulate a question regarding uniform versions of ``large Galois image properties"
for modular abelian varieties of higher dimension, generalizing the well-known case of elliptic curves.
We then answer our question affirmatively in the exceptional image case, and provide lower estimates 
for uniform bounds in the remaining cases. 
\medskip

AMS 2000 Mathematics Subject Classification  11F03 (primary), 11F80, 11G10 (secondary). 
\end{abstract}

\section{Introduction}
\label{Introduction}

Let~$A$ be a simple abelian variety of dimension~$d$ over~$\Q$,  without complex multiplication 
over~$\overline{\Q}$. The images of the Galois representations
$$
\rho_{A,p} \colon G_\Q :=\Gal (\overline{\Q} /\Q )\to \GL (T_p (A) )
$$ 
defined by the $p$-adic Tate modules of~$A$, for~$p$ running through the set of primes, are expected to be 
``generically big". The most famous instance of such a statement is probably Serre's result on elliptic curves~\cite{Se72}, 
which shows that, given a non-CM elliptic curve~$A$ (over any number field) and large enough~$p$ relatively to $A$, the 
image ${\rho}_{A,p} (G_\Q )$ is all of $\GL (T_p (A)) \simeq \GL_2 (\Z_p)$.  

  When~$A$ is a higher dimensional abelian variety, endowed with a polarization~$e$, the existence of the Weil pairing imposes as a first 
constraint that the image of $\rho_{A,p}$ lands in~$\GSP (T_p (A),e)$. For $A$ such that~$\End_{\overline{\Q}} (A)\otimes \Q =\Q$, % (the generic case),  
Serre partly extended his result for elliptic curves: when~$d$ is odd, or~$d=2$ or~$6$, he showed that $\rho_{A,p} (G_\Q )$ 
equals~$\GSP_{2d} (\Z_p )$, for $p$ large enough relative to~$A$ (see~\cite{Se82}, Th\'eor\`eme 3).     
     
  The orthogonal situation, that is, when~$A$ has as many $\Q$-endomorphisms as possible, 
has also been worked out. Recall a simple abelian variety~$A$ over~$\Q$ is said to be {\it of~$\GL_2$-type} 
if~${\End_\Q (A)\otimes \Q}$ is a number field~$E$ of (maximal) degree~$[E:\Q ]=\mathrm{dim} (A)$. A result of Ribet 
(cf. \cite{Ri04}, Theorem 4.4) implies that this is equivalent to $A$ being {\it modular}, i.e., a quotient of some  modular 
jacobian $J_1 (N)$. This result was conditional on Serre's modularity conjecture which has since been proven by 
Khare-Wintenberger and Kisin (see \cite{KW09} and the references therein). Then, as the name indicates, the image of the representation
$$
\rho_{A,p} \colon G_\Q \to \GL (T_p (A)) 
$$ 
lands in $\GL_2 (E \otimes \Q_p )$, and the image inside this last group is still expected to be ``generically big". Indeed, the unavoidable 
constraint on the image given by the prescribed form of the determinant was shown by Momose, Papier and Ribet to be, asymptotically, 
the main one:
\begin{theo} 
\label{RiMo} 
({\em \cite{Ri85}}, Theorem 3.1)  
Let~$f = \sum a_n(f) q^n$ be a non-CM newform of level $N \geq 1$ and weight~${k \geq 2}$. Let~$E =\Q (a_n (f))$ be the number field generated by its Fourier 
coefficients.  %with ring of integers~${\cal O}_E$. 
There is a subfield~$F$ of~$E$, the fixed field of the automorphisms of $E$ corresponding to the extra twists
of~$f$ (so that $F = E$ when $f$ has no extra twists), 
%(so that in particular~${F} =E$ when $f$ has no extra twists), 
and an abelian number field $K$, cut out by the Dirichlet characters corresponding to the extra twists 
of $f$, such that the following holds: for large enough $p$ relative to~$f$, the image of the restriction to $G_K :=  {\mathrm{Gal}}(\overline{\Q} /K)$ 
of the Shimura-Deligne representation~${\rho_{f,p} \colon G_\Q \to \GL_2 ({\cal O}_E \otimes_\Z \Z_p )}$ lands in~$\GL_2 ({\cal O}_{F} \otimes \Z_p)$, 
and is equal to 
$$
\{  u\in \GL_2 ({\cal O}_{F} \otimes_\Z \Z_p ) {\mathrm{\ such\ that\ }} \det (u)\in \Z_p^{*(k-1)}  \} .
$$  
\end{theo}
We content ourselves here with this somewhat imprecise 
statement in order to avoid recalling too many definitions, and we further note that~\cite{Ri85}, 
Theorem 4.1 describes the image $\rho_{f,p} (G_\Q)$ completely. 
Specializing the above result in weight~$k=2$ yields the desired ``big image result" 
for simple non-CM abelian varieties~$A$ of~$\GL_2$-type.

  Now, a natural question is whether or not there are {\it uniform} versions for these large image theorems. Even
more, one might first ask what are the correct questions in the abelian variety setting. Going back again to the case of 
elliptic curves, Serre asked if one can find an absolute constant~$C$ 
(that is, independent of the non-CM rational elliptic curve $A$) such that, for~$p>C$, the representation 
$\rho_{A,p}$ (or, equivalently, the residual representation $\overline{\rho}_{A,p}$) is surjective (see~\cite{Se72}, p.~299 or~\cite{Se81}, p.~199).  

The following classification due to Dickson on finite subgroups of $\PGL_2(\overline{\F}_p)$ (cf. e.g.~\cite{Hu83}, Satz 8.27) allows one 
to break up Serre's question further:
up to conjugacy every finite subgroup~$H$ of $\PGL_2 (\overline{\F}_p )$, for $p$ an odd prime, is contained in either:
\begin{enumerate}
\item an {\em exceptional group}, that is a permutation group isomorphic to~${\gerA}_4$, $\gerS_4$ or~$\gerA_5$;
\item a {\em Borel subgroup} (that is, a finite subgroup of upper triangular matrices);
\item a {\em dihedral group} $D_r$, for some $r \in \N$ not divisible by~$p$;
\item $\PSL_2 (\F_{p^r} )$ or $\PGL_2 (\F_{p^r} )$, for some $r \in \N$.
\end{enumerate}
For elliptic curves $A$, when considering the image of~$\overline{\rho}_{A,p} \subset \GL_2 (\F_{p}) $ (and not its projectivization~$\PPP \overline{\rho}_{A,p}$), it is natural 
to divide case~$(c)$ into two further subcases, depending on wether the image belongs to the normalizer of a {\it split} or {\it nonsplit} Cartan 
subgroup of~$\GL_2 (\F_p )$. More generally, this distinction occurs as soon as there is a natural {\it finite} basefield~$\F_q$ for our 
representation (as will be the case for us): Cartan subgroups in the general case are by definition maximal subtori of~$\GL_n (\F_q )$, and
when~$n=2$ they can only be either split (that is, isomorphic to~$\F_q^* \times \F_q^*$) or nonsplit ($\simeq \F_{q^2}^*$).

   Serre's question for elliptic curves therefore boils down to asking for absolute upper bounds on
primes $p$ such that the images of $\overline{\rho}_{A,p}$, as $A$ varies over all non-CM elliptic curves over $\Q$, 
are contained in groups of the above types: namely, exceptional, Borel, or 
normalizer of split or nonsplit Cartan subgroups. The exceptional cases~$(a)$ are relatively easy to rule out for 
elliptic curves over arbitrary number 
fields (see~\cite{Ma76}, p. 36). For elliptic curves over~$\Q$ it is known that such absolute upper bounds 
also exist in cases~$(b)$ (\cite{Ma78}), and~$(c)$ when the Cartan group is split~(\cite{BP09}), but we still do not know if there are rational non-CM 
elliptic curves with normalizer of nonsplit Cartan structure modulo arbitrarily large~$p$. 
 
   For a higher-dimensional~$A$ which is (and, from this point on, will always be) a simple abelian variety of~$\GL_2$-type over~$\Q$ without 
complex multiplication, set $E:={\End_\Q (A)\otimes \Q}$,
%choose a lattice~$O_p$ in~$E \otimes \Q_p$ (which allows 
%reducing modulo places~$\gerP$ of~$F$ above $p$), 
and consider the residual representation
$$
\overline{\rho}_{A,\gerP} := \rho_{A,p} \!\!\! \mod \gerP \colon G_\Q \to \GL (A[\gerP ] ) =  \GL_2 (\F_{p^n } ),
$$
for $\gerP$ a prime of $E$ above $p$, with residue field $\F_{p^n} = \F_\gerP :={\cal O}_{F,\gerP} /\gerP$. Using Dickson's theorem, we will say a simple and non-CM abelian variety~$A$ 
of~$\GL_2$-type has {\em big image mod~$p$} if the projectivization~$\PPP \overline{\rho}_{A,\gerP}$ of the above~$\overline{\rho}_{A,\gerP}$ 
contain~$\PSL_2 (\F_{p})$ for all primes~$\gerP$ of $E$ of characteristic~$p$. 
Theorem~\ref{RiMo} shows that each~$A$ as above has big image 
mod~$p$, for  $p$ large enough {\it relative to}~$A$. 

   Can this dependence on~$A$ be removed? Keeping in mind that even the simplest situation of rational elliptic curves still remains not 
completely understood, we might nevertheless consider the general case, and the motivation for this note is a very preliminary study of what 
could happen for higher dimensional abelian varieties of~$\GL_2$-type. In this setting, as far as we are aware, virtually nothing is known 
regarding uniform results. As we will see below, if uniform bounds are to be expected they must at least depend on the dimension $d$ of the 
abelian varieties. We feel the most natural statement is the following.  
\begin{Question}
\label{Question1}
Does there exist a function~$B(d)$ for~$d\in \N$ such that, if~$A$ is a rational simple non-CM abelian variety of~$\GL_2$-type with 
dimension less than or equal to $d$, then for all places $\gerP$ of~$E := {\End_\Q (A)\otimes \Q}$, the image of~$\PPP \overline{\rho}_{A,\gerP}$
contains~$\PSL_2 (\F_{p})$ if~${\mathrm{char}} (\gerP )=:p>B(d)$?
\end{Question}
As above, it helps to divide Question~\ref{Question1} according to Dickson's theorem. 
Here, as alluded to above, we speak of split or nonsplit Cartan depending on whether the implicit maximal torus is split or nonsplit over the base 
field~$\F_\gerP$. In the sequel we write $B$ for ``Borel", $E$ for ``exceptional", $SP$ for ``normalizer of split Cartan" and $NSP$ for ``normalizer 
of non-split Cartan".
\begin{Question}
\label{Question2}
For~$*= B, E, SP$ or $NSP$, does there exist a 
function~$B_* (d)$ for~$d\in \N$, such that, if~$A$ is a simple rational non-CM abelian variety 
of~$\GL_2$-type with dimension less than or equal to $d$, then $\overline{\rho}_{A,\gerP}$ cannot have image contained
in a subgroup of type $*$ as soon as ${\mathrm{char}}(\gerP )=:p>B_* (d)$?
\end{Question}
In cases $SP$ and $NSP$, these questions can be reinterpreted in terms of 
Hecke algebras (see Remark~\ref{RemarkFusion} in the last section).

   In this paper we first give an effective positive answer to Question~$\ref{Question2}$ in the easiest case of exceptional groups. We then 
give negative results (i.e., lower bounds for the expected~$B_* (d)$) in the three other cases. The Borel case can be deduced in a 
straightforward way from Ribet's famous work 
\cite{R76a} on the converse to Herbrand's criterion. For the dihedral case, we estimate how other classical results of Ribet on level 
rising~(\cite{Ri83}, or~\cite{Ri90} Theorem 7.3) can be used. We thus produce families of varieties with controlled dimension. But there we also
exhibit a more efficient technique: we show that Hida theory gives further information towards Question~\ref{Question2} in the dihedral case. This last 
technique gives sharper results (than level raising) on the quantitative side, but it does not seem to distinguish easily between the split and nonsplit
subcases. 

   Putting all this together yields the following theorem.
\begin{theo}
\label{mt}
We have:
\begin{enumerate}
\item {\bf (Exceptional subgroups).} Assume $A$ is a rational simple abelian variety of $\GL_2$-type without complex multiplication of 
dimension~$d$, endowed with a Galois structure of exceptional type modulo some prime $p$. Then~$p$ is bounded above in terms of $d$, 
and more precisely, for~$d$ large enough one has 
$$
p \leq 5 \cdot 3^{4d}.
$$ 
\item {\bf (Borel subgroups).}
There is an infinite sequence of prime numbers~$p$ for each of which there is a rational simple abelian variety of $\GL_2$-type 
without complex multiplication endowed with a Borel structure modulo $p$, whose dimension $d$
satisfies
$$
d \leq \frac{(p-5)(p-7)}{24} .
$$
\item {\bf (Dihedral subgroups, I).} 
For each sufficiently large prime~$p$, there is a rational simple abelian variety of $\GL_2$-type 
without complex multiplication, endowed with a normalizer-of-nonsplit-Cartan structure modulo $p$ 
(respectively, a normalizer-of-split-Cartan structure modulo $p$), whose dimension $d$
satisfies
$$
  d \leq C \cdot p^{11/2}  
$$
for some absolute constant~$C$.
\item {\bf (Dihedral subgroups, II).} 
For each large enough prime~$p\equiv 3\mod 4$ there is a rational simple abelian variety of $\GL_2$-type without complex multiplication endowed 
with a projectively dihedral structure modulo $p$,  whose dimension $d$ satisfies
$$
d \leq \frac{(p-5)(p-7)}{24} .
$$
\end{enumerate} 
\end{theo}
Theorem~\ref{mt} has the following consequence for the bound in 
Question~\ref{Question1}:
\begin{coro}
  If a uniform bound~$B(d)$ exists in Question~\ref{Question1}, then it is of order not less 
  than~$O(\sqrt{d})$.
\end{coro}
%
%\begin{coro}
%We have:
%\begin{enumerate}
%\item
%The bound~$B_E (d)$ is finite and bounded above (for large enough~$d$) by
%$$
%B_E (d) \leq 5 \cdot 3^{4d}  .
%$$
%\item
%The bounds~$B_{B} (d)$ and~$B_{SP} (d)$ (finite or not) are asymptotically bounded below, for any~${\varepsilon >0}$, 
%by~${(\sqrt{24}- \varepsilon) \sqrt{d}}$. The same is true for~$B_{NSP} (d)$ with lower bound~$c \cdot d^{1/5.5}$ for %some absolute~$c$.
%\end{enumerate} 
%\end{coro}

%
%\begin{rema}
%\rm
%Some of the results proved here might carry over to the case of modular forms of arbitrary weight, 
%if we could get some control on the variation of the degree of the Hecke fields in a Hida (or Coleman) 
%family, perhaps along the lines of~\cite{Hi11}.
%\end{rema}
% 
%
\section{Exceptional case: Nebentypus and monodromy}
\label{exceptional}
We first prove part~$(a)$ of Theorem~\ref{mt}, which is actually a generalization of a remark of Serre for the case of elliptic curves 
(see~\cite{Ma76}, p. 36). Our first approach is geometric, though later we give a more automorphic proof in a special case. For a more 
computational approach, see~\cite{KV05}.

\subsection{A geometric approach}

With notations as in Theorem~$\ref{mt}$, let~$A$ be a $d$-dimensional abelian variety corresponding to (the Galois orbit of) a 
weight~2 newform $f(q)=\sum_{n>0} a_n q^n$ of some level~$N$. Though this is not necessary, we remark that, up to replacing~$A$ by an 
isogenous variety, one may assume that ${\mathrm{End}}_\Q (A) ={\cal O}_E$ is the full 
ring of integers of the Hecke field $E := \Q(a_n(f))$, with 
the standard notations. Let~$\gerP$ be a prime of~${\cal O}_E$ above~$p$. Let $K$ be a (totally ramified) extension of~$\Q_p$ over which $A$ 
acquires semistable (i.e., possibly good) reduction, with ring of integers~${\cal O}_K$, and absolute Galois group~$G_K$. It is well-known 
that $e := [K:\Q_p ]$ can be bounded in terms of $d$ only. Indeed, by the Galois criterion for semistable 
reduction (\cite{SGA7}, Th\'eor\`eme~IX.3.6 and Proposition~IX.3.5; see also \cite{ST68}, Theorem~1), $A$ acquires good (respectively, bad semistable) 
reduction over any extension field such that the image of inertia at $p$ on the $\ell$-adic Tate module $T_\ell (A)$ is trivial (respectively, 
unipotent of degree 2), for any $\ell \neq p$. If an element~$\gamma$ in~${\mathrm{Aut}} (T_\ell (A))$ is torsion (respectively, has some 
power which is unipotent of degree two) and is trivial mod $\ell >2$, one readily checks that~$\gamma =1$ (respectively, $\gamma$ is 
unipotent of degree two). It follows that one can choose a $K$ as desired inside of $\Q_p (A[\ell ])$. For~$p>3$, we thus have~$e\leq 
{\mathrm{card}} (\prod_{{\cal Q} |3} \GL_2 (\F_{\cal Q} ))<3^{4d}$. This bound is known to be far from sharp -  a (non-CM) elliptic curve over~$\Q$ acquires 
semistable reduction over a number field of degree dividing 6, and similarly, better bounds for jacobian varieties can be found in~\cite{Liu93}, 
Proposition 4 $(\gamma)$.

  Let us still denote by~$A$ what will actually be its N\'eron model over~${\cal O}_K$. The $\gerP$-adic Tate module $A_\gerP := \varprojlim_n A 
[\gerP^n ]$ is ${\cal O}_\gerP$-free of rank 2 (where~${\cal O}_\gerP$ is the completion of the Hecke ring ${\cal O}_E$ at~$\gerP$), and similarly for~$A [\gerP ]
=A_\gerP \otimes {\cal O}_\gerP /\gerP$ over ${\cal O}_\gerP /\gerP$. We define $A^0 [\gerP ]$ to be the connected part of the latter: this can be seen as 
a~{\it $\F_\gerP$-vector space scheme}, for $\F_\gerP = {\cal O}_\gerP/\gerP$, finite and flat over~${\cal O}_K$, or as a certain subspace in the~$p$-torsion 
of the corresponding formal group. It follows from Raynaud's 
classification of group schemes of type $(p,\dots ,p)$ that the tame intertia subgroup of~$G_K$ 
acts on the semi-simplification of $A^0 [\gerP ]$ via products 
of fundamental characters raised to powers in the range $\{ 0,\dots ,e\}$ (\cite{Ra74}, Corollaire 3.4.4). As we are dealing with 
modular representations, those fundamental characters are actually of level (0 or) 1 or 2 (\cite{Se87}, Proposition 1). We now distinguish the 
following four cases: either
\begin{enumerate}
\item  [$(i)$] $A^0 [\gerP ]$ has dimension 0 over $\F_\gerP$, or %${\cal O}_\gerP /\gerP$, or
\item  [$(ii)$] $A^0 [\gerP ]$ has dimension 1 over $\F_\gerP$, or %${\cal O} /\gerP$, or
\item [$(iii)$] $A^0 [\gerP ]$ has dimension 2 over $\F_\gerP$ (so that~$A^0 [\gerP ]=A[\gerP ]$) and the tame inertia acts on its semisimplification via 
fundamental characters of level 2, or
\item [$(iv)$] $A^0 [\gerP ]$ has dimension 2 over $\F_\gerP$, and the tame inertia acts on its semisimplification via fundamental characters of level 1.
\end{enumerate}

  Assume we are in case $(i)$. Then~$A$ corresponds to a weight~2 eigenform $f$ of some level~$N$ and Nebentypus~$\chi$ which is {\it ordinary} 
at $\gerP$. Let $\omega$ be the Teichm\"uller lift of the cyclotomic character and $I_p$ the inertia subgroup at~$p$. Set~${\bar\chi}_p :=\chi |_{I_p} 
\mod {\gerP}$. It is known that $\bar\rho |_{I_p} :=\rho |_{I_p} \mod {\gerP}$ has shape 
$$
\left( \begin{array}{cc}
{\bar\chi}_p \omega & * \\
0 & 1
\end{array} 
\right) .
$$ 

As we assumed $A^0 [\gerP ]$ has dimension 0 over ${\cal O}_\gerP /\gerP$ (that is, $A [\gerP ]$ is \'etale over ${\cal O}_K$) we must 
have~${\bar\chi}_p =\omega^{-1}$. But this implies~$\chi$ has order divisible by~$p-1$, and as~$E:={\End_\Q (A) \otimes \Q}$ contains the 
values of~$\chi$ one has~$\varphi (p-1 ) \le [E:\Q ] ={\mathrm{dim}} (A)=d$ (where~$\varphi$ is Euler's totient function). We know that
$$
\liminf \frac{\varphi(n)\ln(\ln(n))}{n}=e^{-\gamma}
$$
for~$\gamma =0.577...$ the Euler-Mascheroni's constant (see e.g. \cite{Te08}, Th\'eor\`eme 5.6), or more simply that
$$
\varphi (n) > C_\varepsilon \cdot n^{1-\varepsilon}
$$
for any~$\varepsilon >0$ and~$C_\varepsilon >0$ depending on~$\varepsilon$, so that
$$
p \leq c_\varepsilon \cdot d^{1+\varepsilon}
$$
for some~$c_\varepsilon >0$ depending on~$\varepsilon$. This concludes the proof that our case~$(i)$ does not occur for sufficiently large~$p$ 
relatively to~$d$.  

   Assume now we are in case~$(ii)$. Again~$f$ is ordinary at~$\gerP$ and, by Raynaud's  theorem %(\cite{Ra74}, Corollaire 3.4.4)
mentioned above,  
the tame inertia of~$G_K$ at $p$ acts on $A^0 [\gerP ]$ via some power $\omega^a$ of the fundamental character of level 1 so that:
$$
 {\overline{{\rho}}_f} \> \vert_{I_{p}} \simeq 
\left(
\begin{array}{cc}
\omega^a & *\\
0 & 1
\end{array}
\right) 
$$
with~$1\leq a\leq e$. This has projective image containing a cyclic group of order~${(p-1)/{\mathrm{gcd}} (a,p-1)  \geq (p-1)/e}$, which cannot be included 
in an exceptional subgroups if~$p> 5e$. Using our estimate for~$e$ we obtain ${p<5 \cdot 3^{4d}}$, whence our claim in case $(ii)$.
   
   The same argument works for case~$(iii)$. Here Raynaud's result implies that the tame inertia of~$G_K$ at $p$ acts on $A[\gerP ]$ via fundamental 
characters of level 2 raised to powers less than~$e$. As we are dealing with modular forms we know more precisely that this inertia action is via the diagonal 
matrix with coefficients $\omega_2^a$ and $\omega_2^{pa}$, with~$1\leq a\leq e$ and~$\{ \omega_2 ,
\omega_2^p \}$ the conjugate pair of fundamental characters of level 2, which have order $p^2 -1$. Therefore~${\overline{{\rho}}_f} \> \vert_{I_{p}}$ has projective 
image a cyclic group of order~${(p+1)/{\mathrm{gcd}} (a,p+1)}  \geq (p+1)/e$, and we again conclude that $p<5 \cdot 3^{4d}$.

Assume finally we are in case~$(iv)$. Here we will prove that the action of the local Galois group~$G_K$ on $A[\gerP ]$ is not diagonal:
it therefore has a $p$-part, the same is true projectively and this implies the image cannot be exceptional if~$p>5$.    

    Let indeed~$\epsilon$ be the ramification index of~$\gerP$ in the endomorphism ring~${\cal O}_{\gerP}$, and~$\pi$ an uniformizing parameter 
 of~${\cal O}_\gerP$: one has~$\gerP = \pi {\cal O}_\gerP$ and~${p {\cal O}_\gerP =\pi^\epsilon {\cal O}_\gerP}$. Let also~$r=[ \F_\gerP 
:\F_p ]$, for $\F_\gerP :={\cal O}_\gerP /\gerP$, be the residual degree. Let $\cal A$ be  the formal group associated with (the Tate module)~$A_\gerP$, 
and call~$\delta$ its dimension (which is the~$\overline{\F}_p$-dimension of its cotangent space over~$\overline{\F}_p$, or simply the ``number 
of variables" of~$\cal A$). One has $\delta =[{\cal O}_\gerP :\Z_p ] =\epsilon r$. The element~$\pi$ can be seen as belonging to~${\mathrm{End}}_{{\cal O}_K} 
({\cal A})$. The action of~$G_K :={\mathrm{Gal}}(\overline{K} /K)$ commutes with multiplication by $\pi$, so there is a Galois-compatible filtration: 
$$
0\subseteq A[\gerP ] \subseteq A[\gerP^2 ] \subseteq \dots \subseteq A[\gerP^{\epsilon -1} ] \subseteq A[\gerP^\epsilon ]
$$
which can be identified with
$$
0\subseteq \pi^{\epsilon-1} {\cal A}[p ] \subseteq \pi^{\epsilon-2} {\cal A}[p ] \subseteq \dots \subseteq \pi {\cal A}[p ]  \subseteq {\cal A}[p ] 
$$
and each subquotient of the above filtrations is isomorphic, as a $G_K$-module, to~${\cal A}[\pi ]\simeq A[\gerP ]$. One therefore sees that,
whereas $\cal A$ has dimension $\delta$, the group scheme~$A[\gerP ]$ (which has rank~$p^{2r}$) has a cotangent space 
(over~${\overline{\F}}_p$) of dimension~$r=\delta /\epsilon$ (``$r$ variables"). 

   By~\cite{Ra74}, Th\'eor\`eme 3.3.3, one knows that, if $p>e+1$, any (finite flat) group scheme of type $(p,\dots ,p)$ over~${\cal O}_K$ is 
uniquely determined by its generic fiber. Therefore, assuming that the local Galois action is diagonal on~$A[\gerP ] $, the latter 
splits uniquely into the direct sum $A_1 \oplus A_2$ of two finite flat group schemes of rank~$p^r$ over ${\cal O}_K$ (and not only~$K$), each of 
which is an ${\F}_\gerP$-vector scheme of rank one. Considering the geometric special fibers we claim that each~$A_j  \times_{{\cal O}_K} 
{\overline{\F}}_p $ is isomorphic to
\begin{align}
\label{equationsAj}
{\mathrm{Spec}} ({\overline{\F}}_p [X_1 ,\dots ,X_r ]/(X_1^p , \dots , X_r^p )) .
 \end{align}
It indeed follows from~\cite{Ra74} that one can take equations for $A_j$ (over ${\cal O}_K$) of the shape~$X_i^p =\delta_i X_{i+1}$, 
for some~$\delta_i \in {\cal O}_K$ whose valuation verifies $0\le v(\delta_i )\le e$ (cf. {\it loc. cit.}, Corollaire 1.5.1 and p. 266). Moreover, Th\'eor\`eme 3.4.1 
of~\cite{Ra74} tells us that the tame inertia~$I_t$ acts on~$A_j$ via
$$
\psi_1^{v(\delta_r )}  \cdot \psi_2^{v(\delta_1 )} \cdot  \cdots  \cdot\psi_r^{v ({\delta_{r-1}})} ,
$$
with notations of {\it loc. cit.}: $\psi_i :=\psi_1^{p^{i-1}}$, and $\{ \psi_i \}_{1\leq i\leq r}$ is the set of conjugate fundamental characters of level~$r$. On the 
other hand, our running hypothesis~$(iv)$ implies that~$G_K$ acts on the~$A_j$ via some power $\omega^a$ of the fundamental character of level 1, 
which can be expressed in terms of the fundamental characters of level~$r$ as:
$$
\omega =\psi_1^{1+p+p^2 +\dots +p^{r-1}} =\psi_1 \cdots \psi_r.
$$ 
Assuming~$e <p-1$, we therefore see that the~$v(\delta_i )$ are all equal to~$a$ (as $0\leq v(\delta_i )\leq e$). Moreover~$a> 0$ (otherwise the 
equations $X_i^p =\delta_i X_{i+1}$ would show $A_j$ is isomorphic over~${\overline{\F}}_p$ to ${\mathrm{Spec}} ({\overline{\F}}_p [X  ]/(X^{p^r} -cX))$, 
for some $c\neq 0$ in~${\overline{\F}}_p$; the~$A_j$ would therefore be \'etale over~${\cal O}_K$, a contradiction). So~$v(\delta_i )>0$ and Raynaud's 
equations~$X_i^p =\delta_i X_{i+1}$ do give our claim~(\ref{equationsAj}) above.

   But this is not compatible with what we know about~$A[\gerP ]\times_{{\cal O}_K} {\overline{\F}}_p$. For instance,~(\ref{equationsAj}) implies all nilpotent 
functions on~$A_1 \oplus A_2$ are killed by raising to the~$p$th power, so if~$A[\gerP ]\times_{{\cal O}_K} {\overline{\F}}_p$ was split, 
it would in turn be a~${\mathrm{Spec}} ({\cal R})$, with~${\cal R} \otimes_{{\cal O}_K} {\overline{\F}_p}$ a quotient of~${\overline{\F}}_p 
[X_1 ,\dots ,X_r ]/(X_1^p , \dots , X_r^p )$ (recall~$A[\gerP ]$ has~$r$ variables). This would be a contradiction with the fact that~$A[\gerP ]$
has rank $p^{2r}$.  

\medskip  
This proves our claim that~$A [\gerP ]$ is nonsplit as a scheme, from which it follows that the local Galois  projective image has order greater or equal 
to~$p$.  Using again that exceptional subgroups have no $p$-subgroups of order larger than~$5$, this concludes the study of case~$(iv)$, and therefore the 
proof of Theorem~\ref{mt}~$(a)$. 
  
\begin{rema}\rm
A generalization of our arguments might show that we even have the existence of upper 
bounds~$B_E (d,\delta )$ depending on~$d$ and~$\delta$, such that if $A$ is a simple abelian variety of dimension~$d$ which is of $\GL_2$-type 
over a number field~$K$ of degree~$\delta$, and~$A$ is endowed with a Galois structure of exceptional type in prime characteristic~$p$, then
$$
p<B_E (d,\delta )
$$ 
(as is known for elliptic curves). The only problem in generalizing 
our arguments seems to be in proving that the torsion has a non-trivial connected 
part (as a group scheme over~$\overline{\Z}_p$); see case $(i)$ in this paragraph. 
\end{rema}

\begin{rema}\rm
It might also help to briefly recall how things work in the technically simpler setting of elliptic curves. Case~$(i)$ does not occur, as 
the~$p$-torsion is never \'etale in this case. Cases~$(ii)$ and~$(iii)$ can actually be copied with no change. For case~$(iv)$, instead of specializing the 
method above (see, for instance, in~\cite{Mo84}, the proof of Lemma~1.3), one can use Serre's study of the one-parameter formal group 
defined by an elliptic curve (\cite{Se72}, paragraph~1.9 and~1.10). If the elliptic curve in question is supersingular but the tame inertia acts 
via powers of the fundamental character of level 1, it indeed follows from  {\it loc. cit.} that the relevant Newton polygon is broken, and there are 
points in the $p$-torsion of the corresponding formal group which have valuation with denominator divisible by~$p$ (see~\cite{Se72}, 
p.~272). This precisely means that the Galois extension cut out by the $p$-torsion of the elliptic curve has degree divisible by~$p$. Therefore 
the Galois action is non-diagonal, and the projectivization of its image again has a non-trivial~$p$-part. 
This approach might be generalizable to higher-dimensional abelian varieties and formal groups, up to some more technicalities.  
  
   We also note that an elliptic curve over a number field $K$ acquires semi-stable reduction over an 
extension of~$K$ of degree dividing $12$. Therefore, the associated mod~$p$ representations is not projectively exceptional as soon 
as ${p>60[K:\Q ] +1}$: see~\cite{Ma76}, p. 36. (There Mazur asserts that for elliptic curves over~$\Q$, $p>13$ would even do.)    
      
\end{rema} 

\subsection{An automorphic approach, when ${\mathrm{val}}_p (N) \leq 1$}
\label{Conductorp}

For forms of weight 2 and conductor having $p$-adic valuation at most $1$, one can give a purely automorphic proof of part (a) of Theorem 1.4.
This proof has the virtue of appealing to more 
modern technology, but does not cover all cases, since the mod $p$ reductions of forms with high powers of $p$ in the level are not yet fully known.

Suppose $f \in S_2(\Gamma_1(N))$. 
Assume that the power of $p$ dividing $N$ is at most 1.  We show that  $\bar\rho_f$ cannot have exceptional 
projective image, if the dimension of the corresponding abelian variety is bounded, for~$p$ sufficiently large 
depending on the dimension.

First assume that $N$ is prime to $p$. We show that for $p$ sufficiently large, $\bar\rho_f$ on $I_p$ has large projective image irrespective of
the dimension of the underlying abelian variety. Indeed, since $N$ is prime to $p$, then (e.g., see \cite{RS03})
the Serre weight of $\bar\rho_f$ is 2. If $f$ is ordinary at $p$, then it is well-known that on $I_p$,
$\bar\rho_f$ is of the form
\begin{eqnarray*}
  \begin{pmatrix}
    \omega  &   * \\
        0        &  1 
  \end{pmatrix}
\end{eqnarray*}
which has projective image a group of order divisible by $p - 1$. On the other hand, if $f$ is non-ordinary at $p$, then by Fontaine's theorem 
(\cite{Ed92}, Theorem 2.6) (which applies, since the Serre weight $k$ satisfies $2 \leq k \leq p + 1$), $\bar\rho_f$ on $I_p$ has the form 
\begin{eqnarray*}
  \begin{pmatrix}
    \omega_2   &       0  \\  
        0             &      \omega_2^p 
  \end{pmatrix}
\end{eqnarray*}
which has projective image a cyclic group of order $p + 1$. 
If the global projective image of $\bar\rho_f$ is of exceptional
type, then, in either case, this cyclic group cannot have cardinality
larger than 5, so $p \geq 7$ cannot occur.

Now suppose that $p$ exactly divides $N$. Assume $\bar\rho_f$ has exceptional projective image. If there is no Nebentypus at $p$, then $f$ 
is of Steinberg type at $p$, and in particular ordinary, so the projective image on $I_p$ is again of order divisible by $p - 1$,
and so $p \leq 5$. So we may assume that the Nebentypus at $p$ is $\omega^j$ with
$1 \leq j \leq p - 2$. Then, by Proposition 6.18 of~\cite{Sa05}, $\rho_f$ on 
$I_p$ has one of the following three shapes:

\begin{eqnarray*}
  \begin{cases}
    \begin{pmatrix} 
      \omega^{j+1}   &  * \\
         0                    &  1
    \end{pmatrix} & \text{  if } v_p(a_p) = 0, \\
    \begin{pmatrix}
      \omega          &      *   \\
         0                  &   \omega^j
    \end{pmatrix} & \text{  if  } v_p(a_p) = 1, \\
    \begin{pmatrix}
      \omega_2^{j+1}    &   0  \\
         0                        &    \omega_2^{p(j+1)}
    \end{pmatrix} & \text{  if  } 0 < v_p(a_p) < 1. 
  \end{cases}
\end{eqnarray*}
In the first case we see that the order of $\omega^{j+1}$ must be smaller than 5, so that the order of $\omega^j \geq \frac{p-1}{5}$. In 
particular $d \geq \varphi (\frac{p-1}{5} )$. Thus if $p > O( d^{1+\varepsilon} )$, then the exceptional case does not occur. In the second case,
the projective image of $I_p$ on the diagonal is the image of $\omega^{j-1}$ and a similar argument applies. In the 
last case, the projective image of $I_p$ is the image of $\theta := \omega_2^{(p-1)(j+1)}$.
We claim that if this has order at most 5, then $\omega^{j}$ has order at least $\frac{p-1}{3}$, so that we are again done, except in
one case which we treat separately below. Indeed, the order of $\theta$ is $\frac{p+1}{g}$ where $g$ is the greatest common divisor 
of $j+1$ and $p+1$. If the order of $\theta$ is at most 5, then $\frac{p+1}{5} \leq g \leq \frac{p+1}{2}$. Writing $j + 1 = mg$ and $p+1 = ng$, 
for some $1 \leq m < n \leq 5$ (note $j + 1 < p+1$), we have $j + 1 = \frac{m(p+1)}{n}$ for $2 \leq n \leq 5$ and $1 \leq m  <n$, with $(m,n)=1$. An easy 
check shows that for these finitely many values of $j$, the greatest common divisor of $j$ and $p -1$ is at most 3, so that the order of $\omega^j$ is 
at least $\frac{p-1}{3}$, as desired, except when $j+1 = \frac{p+1}{2}$, in which case $\theta$ is quadratic, but $\omega^j$ is also quadratic.

We claim, however that this last subcase cannot occur for sufficiently large~$p$. To see this
suppose that there is an exceptional type form in 
$S_2 (\Gamma_0 (Mp), \chi )$, where $M$ is prime to $p$ and the $p$-part $\chi_p$ of $\chi$
is quadratic. Then consider the Teichm\"uller lift of the 
associated mod $p$ representation (this exists since for
$p > 5$ the mod $p$ representation has order prime to $p$,
since it is an extension of an exceptional group by a subgroup
of scalars of order $p^n - 1$, for some $n$). 
We obtain an odd finite image representation into
${\mathrm GL_2}(W)$, with $W$ is the ring of Witt vectors of the residue
field, which by the recent proof of 
Artin's conjecture (which in turn follows from the proof
due to Khare-Wintenberger/Kisin of Serre's conjecture \cite{KW09}, and 
Khare's proof that Serre implies Artin \cite{Khare}), we know 
comes from a form in 
$S_1 (\Gamma_0 (Mp), \chi')$. Comparing determinants mod $p$  we see that
the $p$-part of $\chi'$ must be
$\omega^{(p+1)/2}$. But an elementary argument 
(see \cite{BG09}, Proposition 5.1) shows
that the Nebentypus at $p$ of an exceptional weight
1 form which is tamely ramified at $p$ must be of bounded
order, which is clearly impossible if $p$ is large. 

This completes the proof.

\section{Borel subgroups and irregular primes}
\label{Borel}
We check part~$(b)$ of Theorem~\ref{mt}. We know from the work of K. L. Jensen (\cite{Je15}; or~\cite{Wa97}, Theorem 5.17) that 
there are an infinite number of irregular primes, that is, primes~$p$ dividing the order of the class group~$C_p$ of~$\Q (\mu_p )$. 
Kummer proved that they are exactly  the primes dividing the numerator of some Bernoulli number~$B_k$ with~$k$ even, $2\leq k\leq 
p-3$, and Herbrand showed more precisely that if~$p$ divides the order of the~$\omega^{1-k}$-isotypic component~$C_p (\omega^{1-k})$ 
of~$C_p$, then~$p$ divides~$B_k$ (where~$\omega$ is as before the cyclotomic character). In his celebrated and 
seminal paper~\cite{R76a} Ribet proved the converse to Herbrand's criterion, and to that end he showed that when~$p\vert B_k$, there 
is a newform in~$S_2^{\mathrm{new}} (\Gamma_0 (p),\omega^{k-2})$ whose associated abelian variety~$A_f$ has a $p$-isogeny. We 
claim those abelian varieties are not of CM type. One way to see this goes as follows. Assuming~$f$ is CM, level considerations show 
that the associated quadratic imaginary field~$K$ has discriminant~$-p$. As~$p$ ramifies in~$K$, the classical theory of complex 
multiplication implies~$f$ is supersingular above~$p$. Now Ribet's representation has 
semi-simplification~$1\oplus \omega^{k-1}$, with~$2\leq k\leq p-3$ (see~\cite{R76a}, Proposition 4.2). Therefore $A_f$
has ($\Z_p$-\'etale subgroups of) $p$-torsion points, and cannot be supersingular at $p$. Invoking finally the fact that
$$
{\mathrm{dim}}(J_1 (p))=(p-5)(p-7)/24
$$ 
(this follows for instance from~\cite{Sh71}, Chapter 2) we conclude the proof. 
\begin{rema}\rm
Of course Ribet's representation shows the existence of $p$-torsion points, not
only a $p$-isogeny, on~$A_f$.
\end{rema}\rm

\begin{rema}\rm
If we can choose $k$ and $p$ such that
$\omega^{k-2}$ has order tending to infinity, then, the above construction gives an infinite sequence
of abelian varieties of $\mathrm{GL}_2$-type of dimension tending to $\infty$, which are 
residually mod $p$ of Borel type, with $p$ tending to $\infty$. 
\end{rema}
\section{Dihedral cases}

\subsection{Using level raising}
We prove part~$(c)$ of Theorem~\ref{mt}.
\begin{propo}
\label{levelraising}
There exists an absolute constant~$C$ such that, for each prime ${p\equiv 3\mod 4}$ (respectively,~$p\equiv 1\mod 4$), there is a non-CM 
abelian variety with a normalizer of nonsplit Cartan (respectively, split Cartan) Galois structure mod $p$ and dimension ${d\leq C\, p^{5.5}}$.
\end{propo}
\begin{proof}
Let $D$ be the discriminant of an imaginary quadratic number field~$K$ with ring of integers~${\cal O}_K$. Let~$A$ be a simple~$\Q$-abelian 
variety of $\GL_2$-type having complex multiplication by~${\cal O}_K$ (take for instance the Weil restriction to~$\Q$ of the Galois conjugacy class 
of elliptic curves over the Hilbert class field~$H_K$ of~$K$) 
%which have complex multiplication by~${\cal O}_K$) 
and let ~$f =\sum_{n} a_n(f) q^n$ the 
CM newform associated with $A$, with conductor~$N_D$. Assume for simplicity that~$f$ has trivial Nebentypus. Let~$p$ be a prime which does 
not divide~$N_D$. For~$\ell$ another prime such that~$\ell \equiv -1\mod p$ and~$\ell$  remains inert in $K$ one has $a_\ell(f) = 0\equiv (\ell +1)
\mod p$, so Ribet's theorem~(\cite{Ri90b}, Theorem 1) shows that~$p$ is a congruence prime between~$f\in S_2^{\mathrm{new}} (\Gamma_0 (N_D ))$ 
and some~$g\in S_2^{\mathrm{new}} (\Gamma_0 (\ell N_D ))$.  Let~$B$ be the abelian variety over~$\Q$ associated with~$g$. Having semi-stable 
bad reduction at~$\ell$, it cannot have complex multiplication. The dimension of~$B$ is bounded above by~${\mathrm{dim}} 
(S_2^{\mathrm{new}} (\Gamma_0 (\ell N_D )))$, which is of shape $\lambda \ell N_D +o(\ell N_D )$ for some~$\lambda\in \Q$ 
(\cite{Sh71}, Chapter 2). Now Linnik's theorem, in the improved explicit version proved by Heath-Brown, shows that one can take 
$\ell \leq c\,  p^{5.5}$, for some constant~$c$ depending only on~$D$ (see e.g.~\cite{IK04}, Theorem 18.1, and the references 
therein). Fix for instance~$D=-4$ in the above, and take for~$A$ the elliptic curve over~$\Q$ with~$j$-invariant~$1728$, conductor~$2^6$, 
having multiplication by the Gaussian integers~$\Z [i]$. The condition~$p\equiv 3 \mod 4$ (respectively,~$p\equiv 1 \mod 4$) insures~$B$ 
has a normalizer-of-{\it nonsplit}-Cartan Galois structure mod $p$ (respectively, normalizer-of-{\it split}-Cartan Galois structure), and
our statement follows. $\Box$  
\end{proof}
\begin{rema}\rm
We understand from~\cite{IK04}, Chapter 18 that the exponent~$5.5$ used here in Linnik's theorem seems to be conjecturally improvable, but not 
under~$2$. We will see in the next section that some elementary Hida theory allows us to produce examples of abelian varieties with projectively dihedral 
Galois structure mod $p$, and dimension bounded above by~$O(p^2 )$. On the other hand, these techniques do not allow
us to distinguish easily between the split and nonsplit cases.  
\end{rema}
\subsection{Using Hida families}
    We first recall a few standard facts on~$\Lambda$-adic modular forms. Fix $p$ an odd prime number, and an embedding~$\overline{\Q} 
\hookrightarrow \overline{\Q}_p$ (which will allow us to think of elements of~$\overline{\Q}$ as living in~$\overline{\Q}_p$). Set~$\Lambda =\Z_p [[X]]$ 
and let~$L$ be the ring of integers of a finite extension of~${\mathrm{Frac}} (\Lambda )$.  An {\it arithmetic point}~$P_{k,\zeta_r}$ is a 
morphism $L \rightarrow \overline{\Q}_p$ of $\Z_p$-algebras extending the algebra homomorphism $\Lambda \rightarrow \overline{\Q}_p$ 
induced  by $X\mapsto \zeta_r (1+p)^{k-1} -1$, with~$2\le k\in \N$ and~$\zeta_r$ a primitive~$p^{r-1}$-th root of unity, $r\geq 1$.
If~${N=N_0 p}$ with~${\mathrm{gcd}} (N_0 ,p)=1$, let~$\psi \colon (\Z /N_0 p\Z )^* \to \overline{\Q}^*$ be a Dirichlet 
character. Let~$\chi_{\zeta_r} \colon (\Z /p^{r} \Z )^* \to \overline{\Q}^*$ be the character which, under the 
decomposition
$$
(\Z /p^{r}\Z )^* \simeq \Z /(p-1)\Z \times \Z /p^{r-1} \Z ,
$$ maps the first factor to $1$ and the generator~$(1+p)$ of the second factor to~$\zeta_r$. 

   A $\Lambda$-adic cusp form of tame character~$\psi$ is a formal series
$$
F(X,q) =\sum_{n\ge 0} a_n (X) q^n  \in L[[X]] 
$$    
such that the specialization~$f_{k,\zeta_r}$ at any arithmetic point~$P_{k,\zeta_r}$ belongs to the space of modular 
forms~$S_k (\Gamma_0 (Np^r ), \psi \omega^{1-k} \chi_{\zeta_r} )$. (This is the weight normalization adapted to ``deformations
of weight 1".) Such a form is said to be 
a~$\Lambda$-adic newform if all its arithmetic specializations are $p$-stablilized $N_0$-newforms.  
   A fundamental theorem of Hida asserts that one can attach to such an eigenform~$F$ a representation
$$
\rho_F \colon \Gal (\overline{\Q} /\Q ) \to \GL_2 ({\mathrm{Frac}} (L))
$$
which can be seen as a family~${P_{k,\zeta_r} (\rho_F )=\rho_{P_{k,\zeta_r (F)}} =\rho_{f_{k,\zeta_r}}}$ of representations 
interpolating those associated by Eichler-Shimura and Deligne to the classical eigenforms~$f_{k,\zeta_r}$ at arithmetic points. 
The weight 1 specializations do give rise to Galois representations too, but they might or might not correspond to classical 
modular forms via Deligne-Serre theory. One checks that the restriction to an inertia group~$I_p$ at~$p$ is of shape
$$
\rho_{F\vert I_p} \simeq \left(
\begin{array}{cc}
\psi \kappa & * \\
0 & 1
\end{array}
\right) 
$$
where~$\kappa$ is the character $\kappa \colon \Gal (\bar\Q / \Q )\twoheadrightarrow \Gal (\Q (\mu_{p^\infty} ) /\Q )
\twoheadrightarrow \Gal (\Q_{\infty} /\Q ) \to \Lambda^*$ which maps the topological generator~$1+p$ of~$1+p\Z_p 
\simeq \Gal (\Q_{\infty} /\Q )$ to $(1+X)\in \Lambda^*$. The mod $p$ representations~$\overline{\rho}_{f_{k,\zeta_r}} 
:=\rho_{f_{k,\zeta_r}} \mod p$ are all isomorphic when irreducible (and, in any case, have the same 
semi-simplification). 

  Let us finally prove part $(d)$ of Theorem~\ref{mt}. Let~$p$ be a prime number equal to $3\mod 4$, so that~$\Q (\sqrt{-p})$, whose ring of 
integers we denote by~${\cal O}_{\Q (\sqrt{-p} )}$, has discriminant $-p$. 
Let~$\alpha_{-p} :=\left( \frac{-p}{\cdot}\right)$ be the corresponding 
quadratic  character. The class number formula shows that $h({\cal O}_{\Q (\sqrt{-p} )})$ is prime to~$p$, as it is bounded above by $\frac{p-1}{2}$. 
The Brauer-Siegel theorem in the case of quadratic imaginary fields actually yields that, for any~$\varepsilon >0$,
$$
p^{1/2 -\varepsilon} < h ({\cal O}_{\Q (\sqrt{-p} )} )< p^{1/2 +\varepsilon} 
$$ 
if~$p$ is large enough, so that in particular~$h({\cal O}_{\Q (\sqrt{-p} )} )$ is 
non-trivial (and, again, prime to $p$) for large enough~$p$. Choose
$$
\Psi \colon \Gal (H(\Q (\sqrt{-p})) /\Q (\sqrt{-p})) \to \overline{\Q}^* \subseteq \overline{\Q}_p^*
$$
a nontrivial character, where $H(\Q (\sqrt{-p}))$ is the Hilbert class field of $\Q (\sqrt{-p})$.
Let
$$
f_\Psi (q)=\sum_{\gerA \subset {\cal O}_{\Q (\sqrt{-p})}} \Psi (\gerA ) q^{N(\gerA )}
$$
be the theta series associated with~$\Psi$. This is a (classical) eigenform of level $p$, weight~$1$ and Nebentypus~$\alpha_{-p}$ 
(cf. e.g.~\cite{Hi93}, paragraph 7.6). The associated Galois representation~$\rho_{f_\Psi}$ is~${\mathrm{Ind}}^\Q_{\Q (\sqrt{-p})} 
(\Psi)$, whose image is included in the normalizer of a Cartan subgroup, but not the Cartan itself (and as~$\Psi$ has prime-to-$p$ 
order, the same is true for~$\overline{\rho}_{f_\Psi} =\rho_{f_\Psi} \mod p$). 
%That Cartan subgroup is well defined, see Remark~\ref{RemarkOrder2} below.
 
  A result of Wiles~(\cite{Wi88}, Theorem 3), generalizing Hida theory for arithmetic points, says that classical eigenforms of weight 1 can also be 
  embedded in $\Lambda$-adic eigenfamilies. Let~$F$ be one such form passing through~$f_\Psi$. This~$F$ 
does not have complex multiplication (i.e., no arithmetic member has complex multiplication), for similar reasons as in paragraph~\ref{Borel}. 
Indeed, if that were the case, a look at the ramification shows that the CM field would have to be~$\Q (\sqrt{-p})$, in which~$p$ ramifies: so the 
weight 2 members of the family would be supersingular, whereas arithmetic specialization of a Hida family are ordinary. Let~$\psi \omega^{1-k} 
\chi_{\zeta_r}$ be the decomposition of the Nebentypus at~$P_{k,\zeta_r}$, using the same notations as in the beginning of this paragraph. 
The tame level~$N_0$ of~$F$ is $1$ and~$\psi$ is some power $\omega^a$ of the Teichm\"uller character~$\omega$. Together with the fact 
that $f_\Psi$ has Nebentypus~$\alpha_{-p}$, we see that $a=(p-1)/2$, which implies that~$P_{2,1} (F)$ is a newform in~$S_2 (\Gamma_0 (p),
\omega^{(p-3)/2} )\subseteq S_2 (\Gamma_1 (p))$. We have therefore built some rational simple abelian variety~$A$ of~$\GL_2$-type, endowed 
with a nontrivial normalizer-of-Cartan structure mod~$p$, which is isogenous to a quotient of~$J_1 (p)$. The shape of the Nebentypus and the 
known dimension of~$J_1 (p)$ give the announced bounds
$$
\varphi \left( \frac{p-1}{2}  \right) \leq {\mathrm{dim}} (A)\leq \frac{(p-5)(p-7)}{24} . \ \ \ \ \ \ \ \ \Box
$$  
\begin{rema}
\label{RemarkOrder2}
\rm
Note that dihedral groups are ambiguously defined in the case when the projective image is the 
Klein 4-group $(\Z /2\Z )^2$, when there are three possible choices for the cyclic subgroup,
but this is not the case for the Galois group built here. Indeed, if it were, one would have two quadratic subextensions of ${\mathrm{Gal}}
(\Q (A[p ])/\Q )$ apart from~$\Q (\sqrt{-p} )$. But the only allowed ramification locus for those number ields is $p$, a contradiction. (Of course,
this yields a proof that $\Q (\sqrt{-p})$ is odd when $p\equiv 3\mod 4$, a classical fact (see e.g. \cite{BS}, Th\'eor\`eme 4 on p. 388).)   
\end{rema} 
\begin{rema}
\label{RemarkFusion}
\rm
Let~$A$ be a~$\GL_2$-type abelian variety as in Question~\ref{Question2}, endowed with a~$SP$ or~$NSP$ Galois structure mod $\gerP$. 
Then $\overline{\rho}_{A,\gerP}$ is the induced representation ${\mathrm{Ind}}_K^{\Q} (\overline{\psi} )$ of some character
$\overline{\psi}$ with values in a finite field $\F$. Let~$\psi \colon \mathrm{Gal} (\overline{\Q} /\Q )\to \overline{\Q}^*$ be
the character deduced from~$\overline{\psi}$ by Teichm\"uller lift. If~$K$ is an imaginary quadratic field (as is necessarily
the case, for instance, for $NSP$ structures when~$A$ is an elliptic curve, as one checks with a look at the image of a complex conjugation) 
then~$\psi$ gives rise to a CM abelian variety~$A_\psi$, whose induced representation~$\rho_{A_\psi ,\gerP}$ is~${\mathrm{Ind}}_K^{\Q} 
(\psi )$. It is known since Hecke that~$A_\psi$ is modular and, as noticed in the Introduction, it follows from Ribet and 
Khare-Wintenberger/Kisin that the same 
is true for $A$. Denoting by~$N$ and~$N_\psi$ the conductors of~$A$ and~$A_\psi$ respectively, one checks that~$N_\psi \vert N$, so 
both abelian varieties can be seen as irreducible components in the spectrum~$\Spec (\T_{\Gamma_1 (N)} )$ of the
full Hecke algebra~$\T_{\Gamma_1 (N)}$ for~$\Gamma_1 (N)$ in weight 2, and by construction those components
intersect in characteristic $p$ (which therefore is a {\it congruence prime} for the cuspforms
$f_A$ and $f_{A_\psi}$ associated to $A$ and $A_\psi$).

   Question~\ref{Question2} then specializes to: if a non-CM irreducible component of some~$\Spec (\T_{\Gamma_1 (N)} )$
intersects a CM one in characteristic~$p$, is it true that the degree~$d$ of the former component has to be such that $B_{SP} (d)$  
or~$B_{NSP} (d)$ are larger than~$p$? We note in passing that sharp bounds for congruence primes of irreducible
components of degree one (that is, elliptic curves) are closely related to deep problems such as modular degree conjecture, 
or $abc$ conjecture (see e.g.%~\cite{Mu96},
~\cite{Fr89} or \cite{Ma95}). Of course, the connectedness of Hecke's algebra spectra 
in weight~2 (``spaghetti principle", see~\cite{Ma76}, Proposition 10.6) shows that some intersection between CM and non-CM components 
has to occur - but {\it quantifying the primes of fusion} is the hard part of the story. We understand that experimental data seem to 
indicate that a large part of fusion occurs in characteristic 2 (cf. e.g. the remark on page 11 of~\cite{MW10}. A typical drawing of the 
situation can be found in~\cite{RS03}, pp. 40-41). 
\end{rema}

\end{document}